\documentclass[a4paper, 12pt]{article}

\usepackage{amssymb, amsmath}

\def\Q{\mathbb{Q}}
\def\Z{\mathbb{Z}}

\def\cP {{\cal P}}

\def\Rc{\hbox{R}}

\def\Ada{\hbox{Ada}}
\def\Mat{\hbox{Mat}}
\def\TaQ{\hbox{TaQ}}

\newcommand{\HRule}{\rule{\linewidth}{0.5mm}}

\newtheorem{defn}{Definition}%[section]

\newtheorem{proposition}[defn]{Proposition}
\newtheorem{theorem}[defn]{Theorem}
\newtheorem{corollary}[defn]{Corollary}
\newtheorem{remark}[defn]{Remark}

           {\vspace{3.3mm}
           \noindent{\bf #1}\it}%
           {\vspace{3.3mm}}

\newenvironment{proof}[1]{
 \trivlist \item[\hskip \labelsep{\it #1}]}{\hfill\mbox{$\square$}
  \endtrivlist}

\title{Linear Solving for Sign Determination}
\author{Daniel Perrucci}

\date{}

\begin{document}

\maketitle

\begin{abstract}
We give a specific method to solve with quadratic complexity the linear systems arising in known algorithms to deal with the sign determination problem. 
In particular, this enable us to improve the complexity bound for sign determination in the univariate case.
\end{abstract}

\section{Introduction}\label{intro}

Let $\Rc$ be a real closed field. A basic problem in computational real algebraic geometry is, given a finite set $Z \subset \Rc^k$ and a finite list $\cP = P_1, \dots, P_s$ of polynomials in $\Rc[X_1, \dots, X_k]$,  to determine the list of sign conditions realized by $\cP$ on $Z$.

A general scheme in most algorithms dealing with the sign determination problem consists on the computation of the Tarski-query (also known as Sturm-query) for $Z$ of many products of the given polynomials, in order to relate through a linear system this quantities with the number of elements in $Z$ satisfying each possible sign condition. 
For instance, the naive approach were the Tarski-query of each of the $3^s$ polynomials of type $P_1^{e_1}\dots P_s^{e_s}$ with $e_i = 0, 1, 2$ for $i = 1, \dots, s$ is computed, leads to a  linear system of size $3^s \times 3^s$.  Nevertheless, if $m = \# Z$ at most $m$ sign conditions will be feasible, and then at most $m$ of the coordinates of the solution of the considered linear system will be different from $0$. 

In \cite{BeKoRe}, the exponential complexity arising from the number of Tarski-query computations and the resolution of the linear system in the approach described above is avoided. This is achieved by means of a recursive algorithm where the list $\cP$ is divided in two sublists, the feasible sign conditions for each sublist is computed, and then this information is combined. 
Such combination is obtained by computing  
at most $m^2$ Tarski-queries and solving a 
linear system of size at most $m^2 \times m^2$. 

In \cite{RoySzp}, \cite{Cann} and \cite[Chapter 10]{BasuPollackRoy}, the methods in \cite{BeKoRe} are further developed. In \cite{RoySzp}, an algorithm is given 
where the number
of points in $Z$ satisfying each feasible sign condition for the list $P_1, \dots, P_i$ is computed sequentially for $i = 1, \dots, s$, following the idea that, at each step, each feasible sign condition for $P_1, \dots, P_{i-1}$ may be extended in at most $3$ ways. 
To deal with the addition of the polynomial $P_i$ to the considered list, 
at most $2m$ Tarski-queries are computed and 
a linear system of size at most $3m \times 3m$ is solved. 
In \cite{Cann}, a more explicit way to choose the  polynomials whose Tarski-query is to be computed is given. In \cite[Chapter 10]{BasuPollackRoy}, also the feasible sign conditions for $P_i$ on $Z$ are computed at step $i$, in order to discard beforehand some non-feasible sign conditions for $P_1, \dots, P_{i}$ on $Z$  extending feasible sign conditions for $P_1, \dots, P_{i-1}$ on $Z$.

Depending on the setting,
the Tarski-queries may be computed in different ways, taking a different number of operations in the field $\Rc$, or in a proper domain ${\rm D}$ containing the coefficients of the polynomials in $\cP$ and polynomials defining the finite set $Z$. As treated with general methods, the linear solving part can be done within $O(m^{2.376})$ operations in $\Q$ (see \cite{CopWin}). In \cite[Section 3]{Cann}, where the univariate case is considered, the cost of linear solving dominates the overall complexity. 
In this work, we fix this situation by giving a specific method to solve with quadratic complexity the linear systems involved (Theorem \ref{Algo_res}). 
Even though this method can be used as a subroutine whenever these systems arise, we emphasize the result of its use in the univariate case. Following the complexity analisis in \cite[Section 3.3]{Cann} we obtain:

\begin{corollary} \label{coro_emp}
Given $P_0, P_1, \dots, P_s \in {\rm R}[X]$, $P_0 \not \equiv 0$,  $\deg P_i \le d$ for $i = 0, \dots, s$, the feasible sign conditions for $P_1, \dots, P_s$ on $\{P_0 = 0\}$ (and the number of elements in $\{P_0 = 0\}$ satisfying each of these sign conditions) can be computed within $O(sd^2 \log^3 d)$ operations in ${\rm R}$. Moreover, if $P_0$ has $m$ real roots, this can be done within $O(smd\log(m)\log^2(d))$ operations in ${\rm R}$.
\end{corollary}

The motivation for this work also comes from \emph{probabilistic} algorithms to determine feasible sign conditions in the multivariate case (\cite{JePeSa}), which produce a 
\emph{geometric resolution} of a set of sample points. In this reduction to the univariate case, 
the degree of the polynomials obtained equals the 
B\'ezout number of some auxiliary polynomial systems, and the complexity of the algorithm depends quadratically on this quantity. Using Corollary \ref{coro_emp}, the treatment of the univariate case to find the feasible sign conditions for the original multivariate polynomials can be done without increasing the overall complexity.

\section{Preliminaries and  Notation}

We will follow mostly the notation in \cite[Chapter 10]{BasuPollackRoy}. In this reference the approach described in Section \ref{intro} is followed with the minor difference that 
the polynomials $P_1, \dots, P_s$ are introduced one at each step from back to front; therefore, the notation is adapted to this order.

For $i = s, \dots, 1$, we call $\cP_i$ the list $P_i, \dots, P_s$,  and, at step $i$, we are given a list $\Sigma = \sigma_1, \dots, \sigma_r$ of elements in $\{0, 1, -1\}^{\cP_i}$ with
$\sigma_1 <_{\hbox{\scriptsize lex}} \dots <_{\hbox{\scriptsize lex}} \sigma_r$ ($0 \prec 1 \prec -1$)
containing, may be properly, all the feasible sign condition for $\cP_i$ on $Z$ and we are to compute the exact list of feasible sign condition for $\cP_i$ on $Z$ and the number of elements in $Z$ satisfying each of these sign conditions. 
Note that the inequality $r \le 3m$ holds at every step.

For $P \in  \Rc[X]$, we note $c(P = 0, Z), c(P > 0, Z)$ and $c(P < 0, Z)$ the number of elements in $Z$ satisfying the condition $P =  0, P > 0$ and $P < 0$ respectively. Recall that the 
Tarski-query of a polynomial $P$ for $Z$ is the number
$$\TaQ(P, Z) = 
c(P > 0, Z) - c(P < 0, Z).$$ 

For $\sigma \in \{0,1, -1\}^{\cP_i}$, we note $c(\sigma, Z)$ the number of elements $x$ in $Z$ satisfying ${\rm sign}(P_j(x)) = \sigma_j$ for $j = i, \dots, s$ and, for a list $\Sigma' = \sigma'_1, \dots, \sigma'_{l}$ of elements in $\{0,1, -1\}^{\cP_i}$, we note $c(\Sigma', Z)$ the vector whose components are $c(\sigma'_1, Z), \dots, c(\sigma'_{l}, Z)$.  We also note $\hat \sigma$ the element of $\{0,1, -1\}^{\cP_{i+1}}$
obtained from $\sigma$ by deleting the coordinate corresponding to $P_i$ and $\hat \Sigma '$ the list $\hat \sigma'_1, \dots, \hat \sigma'_{l}$.

We divide the given list $\Sigma$ in 12 ordered sublists as follows: for $\emptyset \ne B = \{b_1, \dots, b_l\} \subset \{0, 1, -1\}$ with
$b_1 \prec \dots \prec b_l$ 
and for $b \in B$, the list $\Sigma_{b_1 \dots b_l}^{b}$ is composed by those $\sigma \in \Sigma$ such
that
$\sigma(P_i) = b$ and the set $$\{b' \in \{0, 1, -1\} \ | \ \exists \sigma' \in \Sigma \hbox{ such that } \sigma' \hbox{ extends }  \hat \sigma \hbox{ and } \sigma'(P_i) = b'\}$$ equals $B$. For simplicity, we note $\Sigma_0, \Sigma_1$ and $\Sigma_{-1}$ for $\Sigma_0^0, \Sigma_1^1$ and $\Sigma_{-1}^{-1}$ respectively. In addition, since $\hat \Sigma_{b_1 \dots b_l}^{b}$ is the same list for every $b \in B$, we note 
$\hat \Sigma_{b_1 \dots b_l}$ for any such list. 

We also divide the list $\Sigma$ in 3 ordered sublists as follows:
$\Sigma_{(1)}$ is the merge of
lists
$\Sigma_0 , \Sigma_1, \Sigma_{-1},$ $\Sigma_{0 1}^0,$ $\Sigma_{0 -1}^0,$ $\Sigma_{1 -1}^{-1}$ and $\Sigma_{0 1 -1}^0$,
$\Sigma_{(2)}$ is the merge of
lists $\Sigma_{0 1}^1, \Sigma_{0 -1}^{-1}, \Sigma_{1 -1}^{1}$ and $\Sigma_{0 1 -1}^1$ and $\Sigma_{(3)}$ is the same list than $\Sigma_{0 1 -1}^{-1}$.

Consider also the list $\Ada(\Sigma)$ of elements in $\{0,  1, 2\}^{\cP_i}$ (which represents a list of multidegrees) defined recursively by:
$$\left\{ \begin{array}{ll} 0 & \hbox{if } i = s \hbox{ and } r = 1, \cr 0, 1 & \hbox{if } i = s \hbox{ and } r = 2, \cr
0, 1, 2 & \hbox{if } i = s \hbox{ and } r = 3, \cr
0 \times \Ada(\hat \Sigma_{(1)}), \ 1 \times \Ada(\hat \Sigma_{(2)}), \ 2 \times \Ada(\hat \Sigma_{(3)})& \hbox{if } i < s. \cr
\end{array}\right.$$

For a list $A = \alpha_1, \dots, \alpha_{l_1}$ of elements in $\{0,  1, 2\}^{\cP_i}$ and a list $\Sigma' = \sigma'_1, \dots, \sigma'_{l_2}$ of elements in $\{0, 1, -1\}^{\cP_i}$, we note 
$\Mat(A, \Sigma')$ the $\Z^{l_1 \times l_2}$ matrix defined by 
$$\Mat(A, \Sigma')_{j_1j_2} = \sigma'_{j_2}{}^{\alpha_{j_1}},$$
for $j_1 = 1, \dots, l_1$ and $j_2 = 1, \dots, l_2$, with the understanding that $0^0 = 1$.

\begin{remark}
Following this notation, we have that: 
\begin{equation}\label{igualdad}
{\rm Mat}({\rm Ada}(\Sigma), \Sigma) \, c(\Sigma, Z) = {\rm TaQ}(\cP_i^{{\rm Ada}(\Sigma)}, Z)
\end{equation}
and the $r \times r$ matrix ${\rm Mat}({\rm Ada}(\Sigma), \Sigma)$ is invertible (see \cite[Sections 2 and 3]{Cann} or \cite[Proposition 10.65]{BasuPollackRoy}). 
\end{remark}

For a matrix $M$ with rows indexed by
a list $A$ of multidegrees and columns indexed by
a list $\Sigma'$ of sign conditions, and for any sublists $A'$ of $A$ and $\Sigma''$ of $\Sigma'$, we will note, if convenient,
$M_{A'}, M_{\Sigma''}$ and $M_{A', \Sigma''}$ the submatrices
obtained from $M$ by taking only the rows in $A'$, only the columns in
$\Sigma''$, and only the rows in $A'$ and the columns in
$\Sigma''$ respectively. We will use a similar notation for vectors
whose coordinates are indexed by a list of multidegrees or a list of sign conditions.

\section{The specific method for linear solving}

Note that a different order in $\Sigma$ would lead to a permutation of columns in the matrix $\Mat(\Ada(\Sigma), \Sigma)$ and the elements of the vector
$c(\Sigma, Z)$. To explain our method in a simpler way, we will suppose that we change the order in $\Sigma$ in such a way that we find first the elements in $\Sigma_{(1)}$, then those in $\Sigma_{(2)}$ and finally those in $\Sigma_{(3)}$. Nevertheless, this change of order is not actually necessary in the execution of the linear solving algorithm.

If $i = s$, we have that either $r = 1$; $r = 2$ and $\Sigma = 0, 1$;
$r = 2$ and $\Sigma = 0, -1$; $r = 2$ and $\Sigma = 1, -1$ or
$r = 3$. Depending on which of these conditions holds, $\Mat(\Ada(\Sigma), \Sigma)$ is one of the following matrices:
$$\left(\begin{matrix} 1
 \end{matrix}\right), \quad \left(\begin{matrix} 1 & 1 \cr 0  & 1
 \end{matrix}\right),
\quad \left(\begin{matrix} 1 & 1 \cr 0  & -1
 \end{matrix}\right),
\quad \left(\begin{matrix} 1 & 1 \cr 1  & -1
 \end{matrix}\right), \quad
\ \left(\begin{matrix} 1 & 1 & 1 \cr 0 & 1 & -1 \cr 0 & 1 & 1
 \end{matrix}\right).
$$
If $i < s$, consider the matrices $M_1 = \Mat(\Ada(\hat \Sigma_{(1)}), \hat \Sigma_{(1)})$,
$M_2 = \Mat(\Ada(\hat \Sigma_{(2)}),\hat \Sigma_{(2)})$ and
$M_3 = \Mat(\Ada(\hat \Sigma_{(3)}),\hat \Sigma_{(3)})$. Then,
$\Mat(\Ada(\Sigma), \Sigma)$ is the matrix:

$$\left(\begin{array}{c|c|c}
\rule[-6mm]{0.0cm}{1.5cm} \hspace{2mm} M_1 \hspace{2mm} &  M_1' &  M_1''   \cr \hline
\rule[-4mm]{0cm}{1.1cm} X & \hspace{1mm} \tilde M_2 \hspace{1mm}  &  \hspace{-2mm} -M_2' \hspace{-3mm} \cr \hline
\rule[-1mm]{0cm}{0.6cm} Y & Z  &  M_3   \cr
  \end{array}\right)$$

where:
\begin{itemize}
\item $M_1'$ is the matrix formed with the columns of $M_1$ corresponding to sign conditions in $\hat \Sigma_{0 1},
\hat \Sigma_{0 -1}, \hat \Sigma_{1 -1}$ and $\hat \Sigma_{0 1 -1}$,

\item $M_1''$ is the matrix formed with the columns of $M_1$ corresponding to sign conditions in $\hat \Sigma_{0 1 -1}$,

\item $\tilde M_2$ is the matrix obtained from $M_2$ multiplying by $-1$ the columns corresponding to sign conditions in $\hat \Sigma_{0 -1}$,

\item $M_2'$ is the matrix formed with the columns of $M_2$ corresponding to sign conditions in $\hat  \Sigma_{0 1 -1}$,

\item $X = \Mat(\, 1 \times \Ada(\hat \Sigma_{(2)}),\, \Sigma_{(1)}), \
Y = \Mat(\, 2 \times \Ada(\hat \Sigma_{(3)}), \, \Sigma_{(1)})$ and
$Z = \Mat(\, 2 \times \Ada(\hat \Sigma_{(3)}), \, \Sigma_{(2)})$.

\end{itemize}

\begin{remark} \label{relabloq}
\begin{itemize}
 \item The only non-zero columns in matrices 
$X$ and $Y$ are those corresponding to sign conditions in $\Sigma_{1},
\Sigma_{-1}$ and $\Sigma_{1 -1}^{-1}$.
\item The following relations are satisfied:
$$
X_{\Sigma_{1 -1}^{-1}} =  -(M_2)_{\hat \Sigma_{1 -1}}, \quad \quad
Y_{\Sigma_{1 -1}^{-1}} =  Z_{\Sigma_{1 -1}^1}, \quad \quad
Z_{\Sigma_{0 1 -1}^{1}} =  M_3.$$
\item Since $\hat \Sigma_{(3)}$ is included in $\hat \Sigma_{(2)}$, ${\rm Ada}(\hat \Sigma_{(3)})$ is included in ${\rm Ada}(\hat \Sigma_{(2)})$ and then we have:
$$
X_{1 \times {\rm Ada}(\hat \Sigma_{(3)}), \Sigma_{1}} =  Y_{\Sigma_{1}}, \quad
X_{1 \times {\rm Ada}(\hat \Sigma_{(3)}), \Sigma_{-1}} =  -Y_{\Sigma_{-1}}, \quad
X_{1 \times {\rm Ada}(\hat \Sigma_{(3)}), \Sigma_{1-1}^{-1}} =  - Y_{\Sigma_{1-1}^{-1}}.$$
\end{itemize}
\end{remark}

In order to describe our method to solve linear system (\ref{igualdad}), we define the following matrices:

$$
N_1 = \left(\begin{array}{c|c|c}
\rule[-6mm]{0.0cm}{1.5cm} \hspace{1mm} M_1^{-1} \hspace{2mm} &  0 &  0   \cr \hline
\rule[-4mm]{0cm}{1.1cm} 0 & \hspace{2mm} I_2 \hspace{2mm}  &  0 \cr \hline
\rule[-1mm]{0cm}{0.6cm} 0 & 0  &  I_3   \cr
  \end{array}\right) \hspace{-1mm}, \
N_2 = \left(\begin{array}{c|c|c}
\rule[-6mm]{0.0cm}{1.5cm} \hspace{2mm} I_1 \hspace{3mm} &  0 &  0   \cr \hline
\rule[-4mm]{0cm}{1.1cm} -X & \hspace{2mm} I_2 \hspace{2mm}  &  0 \cr \hline
\rule[-1mm]{0cm}{0.6cm} -Y & 0  &  I_3   \cr
  \end{array}\right) \hspace{-1mm}, \
N_3 = \left(\begin{array}{c|c|c}
\rule[-6mm]{0.0cm}{1.5cm} \hspace{2mm} I_1 \hspace{3mm} &  0 &  0   \cr \hline
\rule[-4mm]{0cm}{1.1cm} 0 & \hspace{1mm} M_2^{-1}   &  0 \cr \hline
\rule[-1mm]{0cm}{0.6cm} 0 & 0  &  I_3   \cr
  \end{array}\right) \hspace{-1mm},
$$

$$
N_4 = \left(\begin{array}{c|c|c}
\rule[-6mm]{0.0cm}{1.5cm} \hspace{2mm} I_1 \hspace{3mm} &  0 &  0   \cr \hline
\rule[-4mm]{0cm}{1.1cm} 0 & \hspace{3mm} \tilde I_2 \hspace{2mm}  &  0 \cr \hline
\rule[-1mm]{0cm}{0.6cm} 0 & 0  &  I_3   \cr
  \end{array}\right) \hspace{-1mm}, \
N_5 = \left(\begin{array}{c|c|c}
\rule[-6mm]{0.0cm}{1.5cm} \hspace{2mm} I_1 \hspace{3mm} &  0 &  0   \cr \hline
\rule[-4mm]{0cm}{1.1cm} 0 & \hspace{3mm} I_2 \hspace{2mm}  &  0 \cr \hline
\rule[-1mm]{0cm}{0.6cm} 0 & - \tilde Z  &  I_3   \cr
  \end{array}\right) \hspace{-1mm}, \
N_6 = \left(\begin{array}{c|c|c}
\rule[-6mm]{0.0cm}{1.5cm} \hspace{2mm} I_1 \hspace{3mm} &  0 &  0   \cr \hline
\rule[-4mm]{0cm}{1.1cm} 0 & \hspace{3mm} I_2 \hspace{2mm}  &  0 \cr \hline
\rule[-1mm]{0cm}{0.6cm} 0 & 0 & \hspace{-1mm} M_3^{-1}\hspace{-2mm}   \cr
  \end{array}\right) \hspace{-1mm},
$$

$$
N_7 = \left(\begin{array}{c|c|c}
\rule[-6mm]{0.0cm}{1.5cm} \hspace{2mm} I_1 \hspace{3mm} &  0 &  0   \cr \hline
\rule[-4mm]{0cm}{1.1cm} 0 & \hspace{3mm} I_2 \hspace{2mm}  &  0 \cr \hline
\rule[-1mm]{0cm}{0.6cm} 0 & 0  & \hspace{-1mm}  \frac12 I_3 \hspace{-1mm}   \cr
  \end{array}\right) \hspace{-1mm}, \
N_8 = \left(\begin{array}{c|c|c}
\rule[-6mm]{0.0cm}{1.5cm} \hspace{2mm} I_1 \hspace{3mm} &  0 &  0   \cr \hline
\rule[-4mm]{0cm}{1.1cm} 0 & \hspace{3mm} I_2 \hspace{2mm}  &  I_2' \cr \hline
\rule[-1mm]{0cm}{0.6cm} 0 & 0 &  I_3   \cr
  \end{array}\right) \hspace{-1mm}, \
N_9 = \left(\begin{array}{c|c|c}
\rule[-6mm]{0.0cm}{1.5cm} \hspace{2mm} I_1 \hspace{3mm} &  -I_1' & \hspace{-1mm} -I_1'' \hspace{-1mm}  \cr \hline
\rule[-4mm]{0cm}{1.1cm} 0 & \hspace{3mm} I_2 \hspace{2mm}  &  0 \cr \hline
\rule[-1mm]{0cm}{0.6cm} 0 & 0 & \hspace{-1mm} I_3\hspace{-2mm}   \cr
  \end{array}\right) \hspace{-1mm},
$$

where $I_1, I_2$ and $I_3$ denote the identity matrices which size is the length of $\Sigma_{(1)}, \Sigma_{(2)}$ and $\Sigma_{(3)}$ respectively and, if we index the columns of
$I_1, I_2$ and $I_3$ with $\hat \Sigma_{(1)}, \hat \Sigma_{(2)}$ and $\hat \Sigma_{(3)}$:

\begin{itemize}
\item $\tilde I_2$ is the matrix obtained from $I_2$  multiplying by $-1$ the columns corresponding to sign conditions in $\hat \Sigma_{0 -1}$ and by $\frac12$ the columns corresponding to sign conditions in $\hat \Sigma_{1 -1}$,

\item $\tilde Z$ is the matrix obtained from $Z$  multiplying by $0$ the columns corresponding to sign conditions in $\Sigma_{1 -1}^{1}$,

\item $I_2'$ is the matrix formed with the columns of $I_2$ corresponding to sign conditions in $\hat \Sigma_{0 1 -1}$,

\item $I_1'$ is the matrix formed with the columns of $I_1$ corresponding to sign conditions in
$\hat\Sigma_{0 1},  \hat\Sigma_{0 -1}, \hat\Sigma_{1 -1}$ and $\hat\Sigma_{0 1 -1}$,

\item $I_1''$ is the matrix formed with the columns of $I_1$ corresponding to sign conditions in
$\hat\Sigma_{0 1 -1}$.
\end{itemize}

\begin{proposition} \label{descompo}
The matrix ${\rm Mat}({\rm Ada}(\Sigma),\Sigma)^{-1}$ equals the product $N_9\dots N_1$.
\end{proposition}

\begin{proof}{Proof:} First, note that 
$$M_1^{-1}M_1' = I_1', \quad \quad \quad
M_1^{-1}M_1'' = I_1'', \quad \quad \quad
M_2^{-1}M_2' = I_2'.$$
Because of the first item of Remark \ref{relabloq} we can conclude that
$$XI_1'' = 0,   \quad \quad \quad YI_1'' = 0,$$
and using the first and second items of Remark \ref{relabloq}, we can conclude  that
$$-XI_1' + \tilde M_2 = \dot M_2, \quad \quad \quad 
-YI_1' + Z = \tilde Z, \quad \quad \quad 
ZI_2' = M_3,$$
where $\dot M_2$ is the matrix obtained from $M_2$ multiplying by $-1$ the columns corresponding to sign conditions in $\hat \Sigma_{0 -1}$ and by $2$ the columns corresponding to sign conditions in $\hat \Sigma_{1 -1}$.
With all these relations, the proof can be done by simple computation of the product $N_9\dots N_1\Mat(\Ada(\Sigma),\Sigma)$.

\end{proof}

The decomposition of $\Mat(\Ada(\Sigma),\Sigma)^{-1}$ as a product of several matrices in Proposition \ref{descompo} leads us to the following recursive algorithm:

\bigskip

\HRule

\textbf{Algorithm}:  Auxlinsolve

Input: $\Sigma, \TaQ(\cP_i^{{\rm Ada}(\Sigma)}, Z)$.

Output: $c(\Sigma, Z)$.

Procedure: \begin{enumerate}
\item If $i = s$, return $\Mat(\Ada(\Sigma), \Sigma)^{-1}\TaQ(\cP_i^{{\rm Ada}(\Sigma)}, Z)$.

\item If $i < s$:

\begin{enumerate}

\item[0.] Initialize $c = \TaQ(\cP_i^{{\rm Ada}(\Sigma)}, Z)$.

\item[1.] $c_{\Sigma_{(1)}} =$ Auxlinsolve$(\hat \Sigma_{(1)} , c_{\Sigma_{(1)}})$.

\item[2.] 
$c_{\Sigma_{(2)}} = c_{\Sigma_{(2)}} -
\Mat(1 \times \Ada(\hat\Sigma_{(2)}), \Sigma_{1})c_{\Sigma_{1}} -
\Mat(1 \times \Ada(\hat\Sigma_{(2)}), \Sigma_{-1})c_{\Sigma_{-1}} - \Mat(1 \times \Ada(\hat\Sigma_{(2)}), \Sigma_{1-1}^{-1})c_{\Sigma_{1-1}^{-1}};$

$c_{\Sigma_{(3)}} = c_{\Sigma_{(3)}} -
\Mat(2 \times \Ada(\hat\Sigma_{(3)}), \Sigma_{1})c_{\Sigma_{1}} -
\Mat(2 \times \Ada(\hat\Sigma_{(3)}), \Sigma_{-1})c_{\Sigma_{-1}} - \Mat(2 \times \Ada(\hat\Sigma_{(3)}), \Sigma_{1-1}^{-1})c_{\Sigma_{1-1}^{-1}}.$

\item[3.] $c_{\Sigma_{(2)}} =$ Auxlinsolve$(\hat \Sigma_{(2)} , c_{\Sigma_{(2)}})$.

\item[4.] $c_{\Sigma_{0-1}^{-1}} = - c_{\Sigma_{0-1}^{-1}};$ 

$c_{\Sigma_{1-1}^1} = \frac12 c_{\Sigma_{1-1}^1}$.

\item[5.]  $c_{\Sigma_{(3)}} = c_{\Sigma_{(3)}} -
\Mat(2 \times \Ada(\hat\Sigma_{(3)}), \Sigma_{01}^1)c_{\Sigma_{01}^1} -
\Mat(2 \times \Ada(\hat\Sigma_{(3)}), \Sigma_{0-1}^{-1})c_{\Sigma_{0-1}^{-1}}
- M_3c_{\Sigma_{01-1}^{1}}$.

\item[6.] $c_{\Sigma_{(3)}} =$ Auxlinsolve$(\hat \Sigma_{(3)} , c_{\Sigma_{(3)}})$.

\item[7.]  $c_{\Sigma_{(3)}} = \frac12 c_{\Sigma_{(3)}}$.

\item[8.]  $c_{\Sigma_{01-1}^1} = c_{\Sigma_{01-1}^1} + c_{\Sigma_{(3)}}$.

\item[9.] $c_{\Sigma_{01}^0} = c_{\Sigma_{01}^0} - c_{\Sigma_{01}^1}$;

$c_{\Sigma_{0-1}^0} = c_{\Sigma_{0-1}^0} - c_{\Sigma_{01}^{-1}}$;

$c_{\Sigma_{1-1}^{-1}} = c_{\Sigma_{1-1}^{-1}} - c_{\Sigma_{1-1}^{1}}$; 

$c_{\Sigma_{01-1}^0} = c_{\Sigma_{01-1}^0}
- c_{\Sigma_{01-1}^1} - c_{\Sigma_{(3)}}$.

\item[10.] Return $c$.
\end{enumerate}
\end{enumerate}

\HRule

\bigskip

\begin{theorem} \label{Algo_res}
Algorithm Auxilinsolve solves linear system (\ref{igualdad}) within  $2r^2$
operations in $\Q$.
\end{theorem}

\begin{proof}{Proof:}

The correctness of the algorithm follows since, for $j = 0, \dots, 9$, after Step $2.j$ we have computed
$$c = N_j\dots N_1\Mat(\Ada(\Sigma), \Sigma)c(\Sigma, Z).$$

To bound the number of operations needed, we proceed by  induction on $i$. If $i = s$, the result follows, since the inverse of the 5 possible matrices $\Mat(\Ada(\Sigma), \Sigma)$ is pre-computed and the product by $\TaQ(\cP_s^{{\rm Ada}(\Sigma)}, Z)$ takes $r(2r-1)$ operations in $\Q$.

Suppose now $i < s$.
For $\emptyset \ne B = \{b_1, \dots, b_l\} \subset \{0, 1, -1\}$
note by $r_{b_1 \dots b_l}$ the size of the list $\Sigma_{b_1 \dots b_l}^{b}$ for any $b \in B$. Using the inductive hypothesis, the number of operations in each step  is bounded in the following way:
\begin{enumerate}
\item[2.1.] $2(r_0 + r_1 + r_{-1} + r_{0 1} + r_{0 -1} + r_{1 -1} + r_{0 1 -1})^2$.
\item[2.2.] $2(r_{01} + r_{0-1} + r_{1-1} + 2r_{01-1})(r_{1} + r_{-1} + r_{1-1})$.
\item[2.3.] $2(r_{0 1} + r_{0 -1} + r_{1 -1} + r_{0 1 -1})^2$.
\item[2.4.] $r_{0 -1} + r_{1 -1}$.
\item[2.5.] $2r_{01-1}(r_{01} + r_{0-1} + r_{01-1})$.
\item[2.6.] $2r_{01-1}^2$.
\item[2.7.] $r_{01-1}$.
\item[2.8.] $r_{01-1}$.
\item[2.9.] $r_{0 1} + r_{0 -1} + r_{1 -1} + 2r_{01-1}$.

\end{enumerate}

Since the sum of all this numbers is always lower than or equal to $2r^2 = 2(r_0 + r_1 + r_{-1} + 2r_{0 1} + 2r_{0 -1} + 2r_{1 -1} + 3r_{0 1 -1})^2$, the result follows.

\end{proof}

\begin{remark}
The third item of Remark \ref{relabloq} implies that Step 2.2 could be replaced in the following way:
\begin{enumerate}
\item[{\rm 2.2.'}] $v = {\rm Mat}(1 \times {\rm Ada}(\hat\Sigma_{(2)}), \Sigma_{1})c_{\Sigma_{1}};$

$v' = {\rm Mat}(1 \times {\rm Ada}(\hat\Sigma_{(2)}), \Sigma_{-1})c_{\Sigma_{-1}};$

$v'' = {\rm Mat}(1 \times {\rm Ada}(\hat\Sigma_{(2)}), \Sigma_{1-1}^{-1})c_{\Sigma_{1-1}^{-1}};$

$c_{\Sigma_{(2)}} = c_{\Sigma_{(2)}} - v - v' - v'';$

$c_{\Sigma_{(3)}} = c_{\Sigma_{(3)}} - v_{1\times {\rm Ada}(\Sigma_{(3)})} + v'_{1\times {\rm Ada}(\Sigma_{(3)})} + v''_{1\times {\rm Ada}(\Sigma_{(3)})}.$
\end{enumerate}
which takes a smaller number of operations than Step 2.2.

\end{remark}

\end{document}